\documentclass[11pt,leqno]{amsart}

\usepackage[american]{babel}
\usepackage{amsmath,amsthm,mathrsfs}
\usepackage{amsfonts, latexsym, amssymb}
\usepackage{hyperref}
\usepackage{graphics,color}
\usepackage{graphicx,xcolor}

\usepackage{tikz}
\usetikzlibrary{patterns}

\usepackage[a4paper,asymmetric]{geometry}

\newcommand{\R}{\mathbb{R}}

\newcommand{\rn}{\mathbb{R}^N}
\newcommand{\aaa}{\textsf{a}}
\newcommand{\bd}{\partial}

\newcommand{\ep}{\varepsilon}

\newcommand{\la}{\lambda}

\newcommand{\Om}{\Omega}

\newcommand{\p}{\partial}

\newcommand{\xx}{\bar{x}}

\newtheorem{theorem}{Theorem}[section]
\newtheorem{proposition}[theorem]{Proposition}
\newtheorem{corollary}[theorem]{Corollary}
\newtheorem{lemma}[theorem]{Lemma}
\newtheorem{remark}[theorem]{Remark}

\newtheorem{claim}[theorem]{Lemma}

\usepackage{amsmath}
\begin{document}

\author[C. Bianchini]{Chiara Bianchini}
\author[A. Henrot]{Antoine Henrot}
\author[P. Salani]{Paolo Salani}

\address{C. Bianchini, Dip.to di Matematica ``U. Dini'', Universit\`a degli Studi di Firenze, Viale Morgagni 67/A, 50134 Firenze - Italy}
\email{cbianchini@math.unifi.it}

\address{A. Henrot: Institut Elie Cartan, UMR CNRS 7502, Universit\'e de Lorraine, 
Boulevard des Aiguillettes B.P. 70239, F-54506 Vandoeuvre-les-Nancy Cedex, France}
\email{antoine.henrot@univ-lorraine.fr}

\address{P. Salani, Dip.to di Matematica ``U. Dini'', Universit\`a degli Studi di Firenze, Viale Morgagni 67/A, 50134 Firenze - Italy}
\email{salani@math.unifi.it}
\date{}


\begin{abstract}
We investigate an overdetermined Torsion problem, with a non-constant positively homogeneous boundary constraint on the gradient.
We interpret this problem as the Euler equation of a shape optimization problem, we prove existence and regularity of a solution. 
Moreover several geometric properties of the solution are shown.
\end{abstract}

\title[An overdetermined problem]{An overdetermined problem with non constant boundary condition}

\maketitle

\section{Introduction}
An important class of shape optimization problems occurs when a free boundary problem is considered.
Of particular interest is the case of {overdetermined} boundary value problems, which in general corresponds to a classical partial differential equation where both Dirichlet and Neumann boundary conditions are imposed on the boundary of the domain. 
Obviously this over-determination makes the domain itself unknown. 
Interesting questions are then the proof of the existence of a solution, possibly uniqueness and the study of qualitative properties of a solution.
A very large amount of literature exists for such problems, depending on the governing operator 
and on the overdetermined conditions which in many cases writes as $u=0$ and $|\nabla u|=$constant on $\p\Om$, although several
other kinds of overdetermined conditions and operators have been considered in the literature (see for instance \cite{AM}, \cite{CS}, \cite{CMS1}, \cite{Greco2}, \cite{HPh}, \cite{Salani}, \cite{Sc}, \cite{Sh} and references therein).

In particular  here we will deal with a governing operator of the torsion type, that is with the most classical equation:
\begin{equation}\label{torsioneq}
-\Delta u =1 \qquad\text{ in }\Om
\end{equation}
and of course everybody knows the famous result by Serrin \cite{Serrin} who proved that if a solution to (\ref{torsioneq}) exists with $u=0$ and $|\nabla u|=$ constant on the boundary of $\Om$, then the set $\Om$ must be a ball and the function $u$ is radial.
After Serrin, many authors investigated such a problem and stability results hold.
More precisely it has been proved in \cite{ABR} and in \cite{BNST} that if $|\nabla u|$ is almost constant on the bounday of $\Om$, then $\Om$ is not far from being a ball.
A natural question is then to investigate what happens in the case of a ``genuine'' non-constant boundary condition for the gradient.
In particular here we are interested in an overdetermined condition of the type 
$$|\nabla u(x)|=g(x)\quad \text{on }\partial\Omega\,.
$$

Let us quote that the same conditions have been already considered  for differential problems of the torsion and of the Bernoulli types  in \cite{A1, A2, AlCa, GuSh, HeSh}.
On the other hand in \cite{Bia}  the Bernoulli interior problem equipped with overdetermined condition written as $|\nabla u(x)|=g(\nu(x))$ is considered, where $\nu$ is the outer unit normal of the free boundary.
At our knowledge an analogous overdetermined condition (which would be very natural) for the torsion problem has not been considered yet.

Let us now describe in detail the problem we are interested in; we will relate it below to the existing literature.

For any bounded open set $\Om$ (or set of finite measure)
we denote by $u_\Om$ the \emph{stress function} of $\Om$; that is the solution of the \emph{torsion problem}: 
\begin{equation}\label{i1}
\begin{cases}
-\Delta u_\Om=1 & \qquad \text{in }\;\Om\\
u_\Om=0 & \qquad \text{on }\;\p\Om,
\end{cases}
\end{equation}
or of its weak form
\begin{equation}\label{i2}
u_{\Omega}\in H^1_0(\Omega),\;\;\forall\, v\in H^1_0(\Omega): \ \int_\Om\nabla u_{\Omega}\nabla v=\int_\Om u_{\Omega}\,v,
\end{equation}
where $H^1(\Omega)$ is the Sobolev space of functions in $L^2(\Omega)$ whose first derivatives are in $L^2(\Omega)$ and
$H^1_0(\Omega)$ is the closure in $H^1(\Omega)$ of smooth functions compactly supported in $\Omega$.
Notice that the stress function $u_\Om$ can be characterized as
\begin{eqnarray}\label{i3}
 G_\Om(u_\Om)&=&\min\{G_\Om(v),\ v\in H^ 1_0(\Om)\}\;\text{ where}\\
  G_\Om(v)&=&\frac{1}{2}\int_\Om |\nabla v|^2\,dx - \int_\Om v\,dx\,.\nonumber
\end{eqnarray}
Let $g$ be a function defined on $\rn$ and satisfying
\begin{equation}\label{assum_g}
\begin{cases}
\text{$g:\rn \to \R$ positively homogeneous of degree $\alpha$}\\ 
\qquad\qquad\text{(i.e. $g(tx)=t^\alpha g(x)$ $\forall t>0,\forall x\in\rn$),}\\
\text{$g$ H\"older continuous, $g>0$ outside $0$.}
\end{cases}
\end{equation}
We are interested in solving the following overdetermined free
boundary problem of the torsional type with a non constant boundary condition:
\begin{equation}\label{ovfbp1}
\begin{cases}
-\Delta u_\Om=1 & {\rm in }\;\Om\\
u_\Om=0 & {\rm on }\;\p\Om\\
|\nabla u_\Om|=g(x) & {\rm on }\;\p\Om.
\end{cases}
\end{equation}
In this context, this problem is close to the one considered by B. Gustafsson
and H. Shahgholian in \cite{GuSh}. 
In fact they study the partial differential equation $-\Delta u=f$
where $f$ is a function (or a measure) {\it whose positive part $(f)^+$ has compact support} (here and later we will denote 
by $(f)^+$ the positive part of the function $f$, that is $\max\{0,f\}$). 
This
makes a real difference with (\ref{ovfbp1}) as will be clear in a while.
Indeed they use the Alt-Caffarelli approach, consisting in minimizing
\begin{equation}\label{acf}
 G_\Omega(u_\Om)+\frac{1}{2}\,\phi(\Om),
\end{equation}
where 
\begin{equation}\label{defphi}
\phi(\Om):=\int_\Om g^2(x)\,dx
\end{equation}
(they write it as a problem in the calculus of variations replacing $\Om$
by $\{u>0\}$ but it does not change anything). 
Unfortunately, in our case, the fact that the support of $(f)^+\equiv 1$ is the whole $\rn$ makes the minimization problem
(\ref{acf}) in general ill-posed, since the infimum can be $-\infty$ as it is easily
seen by explicit computations in the radial case. 
This is the reason for we have chosen a different method. 

We use a variational approach which consists in looking at (\ref{ovfbp1}) as the optimality conditions of some shape 
optimization problem. 
More precisely, let $J$ be the functional defined as the opposite of the torsional rigidity:
\begin{equation}\label{s1}
J(\Om)=\frac{1}{2}\int_\Om |\nabla u_\Om|^2\,dx - \int_\Om u_\Om\,dx=
-\frac{1}{2} \int_\Om u_\Om\,dx=-\frac{1}{2} \int_\Om |\nabla u_\Om|^2\,dx.
\end{equation}
and consider $\phi$ as defined in (\ref{defphi}).
The shape optimization problem we are interested in, consists in minimizing $J$ with the constraint 
\begin{equation}\label{s2}
\phi(\Om):=\int_\Om g^2(x)\,dx \leq 1.
\end{equation}
Let us point out that this introduces a further difficulty since we have to deal with a Lagrange multiplier. 
The choice of a homogeneous function $g$, allows us to encounter this difficulty  (see the proof of Corollary 
\ref{maincor}) since it permits to estimate the value of the Lagrange multiplier.
However we point out that the existence of a solution to the shape optimization problem is guaranteed under the simple assumption $g(x)>0$ outside the origin and $\lim g(x)=+\infty$ for $|x|\to +\infty$.

We remark that  this shape optimization problem is a variant of the famous Saint-Venant problem, and hence it has its own practical interest. 
In the classical Saint-Venant problem one looks for the shape of the set with given area which has maximal torsional 
rigidity; in \cite{Pol} G. Poly\`a proved that the answer is the ball. 
Here we solve the same problem in the class of non-uniformly dense sets, whose density depends on the function $g$. 

\medskip
The paper is organized as follows. In Section \ref{secSop} we prove the main results of the paper:
existence and regularity of a minimizer for our shape optimization problem and, as a consequence,
the existence of a solution to the free boundary problem. In Section \ref{basicpropSec} we prove
some basic properties: the origin is in general inside the solution, monotonicity with respect to $g$ and uniqueness of the solution when $\alpha>1$. 
In Section \ref{starshapeSec} we investigate starshape and we prove the starshapedness of solutions for $\alpha\neq 1$. Section \ref{convexitySec}
is devoted to prove the
 convexity of the solution, under suitable assumptions. 
In Section \ref{symmetrysec} we prove some symmetry results and we study the stability of the radiality when
$g$ is close to be radially symmetric.
Finally in Section \ref{finalSec} we investigate the relationship between the sahpe of the solution and the shape of the level sets of $g$, giving some a priori bounds for the solution $\Omega$ in terms on $G_1=\{x\in\R^n\,:\, g(x)\leq 1\}$.

\section{The shape optimization problem}\label{secSop}

We consider the energy functional $J(\cdot)$ defined in (\ref{s1}). 
We recall that, by the maximum principle, $J$ is decreasing with respect
to set inclusion that is: $\Om_1\subset \Om_2$ implies $J(\Om_1)\geq J(\Om_2)$.

In this section, we want to minimize the functional $J(\Om)$ among open sets
satisfying (\ref{s2}).
Let us remark that the measure of sets $\Om$ satisfying (\ref{s2}) is bounded
since $g(x)$ diverges at infinity for homogeneity.

As already emphasized, the homogeneity of $g$ plays  a crucial role. 
In particular it makes the problem having a nice behaviour with respect to homotheties.
More precisely, for every $t>0$, $\Om\subset\rn$, it holds 
\begin{eqnarray*}
J(t\Omega)   &=& -\frac{1}{2}\,\int_{t\Omega} t^2 u_\Omega(x/t)\,dx=t^{2+N}J(\Omega),\\ \label{hom2}
\phi(t\Omega)&=& \int_{t\Omega} g^2(x)\,dx=t^{2\alpha+N}\phi(\Omega), \label{hom1}
\end{eqnarray*}
where the first equality follows from the fact that the stress function of $t\Omega$ is 
$$
u_{t\Om}(x) = t^2 u_\Omega(\frac xt).
$$

Therefore the two following problems are equivalent from a qualitative point of view:
\begin{eqnarray}
\min_{\Om\subset\rn}\{J(\Omega) :  \phi(\Omega)\leq 1\} \label{min1},\\
\min_{\Om\subset\rn} \phi(\Omega)^{-\frac{2+N}{2\alpha+N}}J(\Omega).\label{min2}
\end{eqnarray}

Let us state our main result on existence and regularity of a solution to Problem (\ref{min1}) or (\ref{min2}), whose proof is given in Section \ref{secmainthm}.
In particular we stress that the existence part in Theorem \ref{mainthm} does not need the homogeneity property of $g$ but it holds under the weaker assumption $g>0$ outside the origin and $g$ tends to infinity for $|x|\to\infty$.
\begin{theorem}\label{mainthm}
Problem (\ref{min1}) admits a solution $\Om$. 
This one is $C^{1,\beta}$ in dimension $N=2$. 
In dimension $N\geq 3$, the reduced boundary $\p_{red} \Om$ is $C^{1,\beta}$ and $\p\Om\setminus \p_{red} \Om$ has 
zero $(N-1)$-Hausdorff measure.
\end{theorem}
The existence of a solution to the overdetermined Free Boundary Problem (\ref{ovfbp1}) follows:
\begin{corollary}\label{maincor}
Let $g$ satisfy (\ref{assum_g}) for some $\alpha>0$, $\alpha\not= 1$. 
Then there exists a solution to the overdetermined Free Boundary Problem (\ref{ovfbp1}).
\end{corollary}
\begin{remark}
The overdetermined boundary condition $|\nabla u_\Omega(x)|=g(x)$ holds on the regular
part of the boundary of $\Omega$. 
\end{remark}
\begin{proof}[Proof of Corollary \ref{maincor}] Let $\Om$ be a solution of the shape optimization problem (\ref{min1}) given by
Theorem \ref{mainthm}. 
Since the reduced boundary of $\Om$ is $C^{1,\beta}$,
by classical regularity results, the gradient of $u_\Om$ is defined almost
everywhere on the boundary. 
We can then write the optimality conditions.
For that purpose, we use the classical shape derivative as defined, for
example in \cite[chapter 5]{HP}. The derivative of the functional $J$
at $\Om$ in the direction of some deformation field $V$ is given
by
\begin{equation}\label{der1}
    dJ(\Om;V)=-\int_\Om \nabla u_\Om.\nabla u' -\frac{1}{2}\,\int_{\p\Om}
    |\nabla u_\Om|^2 \langle V,n\rangle,
\end{equation}
where $u'$ is the derivative with respect to the domain of $u_\Om$, solution of
\begin{equation}\label{der2}
\begin{cases}
-\Delta u'=0 & {\rm in }\;\Om\\
u'=-\frac{\p u_\Om}{\p n}\,\langle V,n\rangle & {\rm on }\;\p\Om.
\end{cases}
\end{equation}
By Green formula and (\ref{der2}), $\int_\Om \nabla u_\Om.\nabla u'=0$.
On the other hand, the derivative of the constraint $\phi$ is
\begin{equation}\label{der3}
    d\phi(\Om;V)=\int_{\p\Om} g^2 \langle V,n\rangle.
\end{equation}
By the optimality condition there exists a Lagrange multiplier
$\mu$ such that, for any deformation field $V$, we have 
$$
dJ(\Om;V)=\mu\ d\phi(\Om;V),
$$ 
and, according to (\ref{der1}), (\ref{der3}) this writes as
\begin{equation}\label{der4}
     -\frac{1}{2}\,\int_{\p\Om} |\nabla u_\Om|^2 \langle V,n\rangle=\mu \int_{\p\Om} g^2 \langle V,n\rangle.
\end{equation}
Since equality (\ref{der4}) must hold for any $V$, we get
$$|\nabla u_\Om|^2 = -2\mu g^2\quad \mbox{on $\p\Om$}.$$
Let us remark that $\mu$ cannot be zero by unique continuation property
(or Hopf's lemma). 
Now, replacing $\Om$ by $t\Om$ where 
$$
t=(-2\mu)^{\frac 1{2(\alpha-1)}},
$$
and taking into account that 
$$|\nabla u_{t\Om}(x)|=t|\nabla u_{\Om}(x/t)|= t(-2\mu)^{1/2}g(x/t)=t^{1-\alpha}(-2\mu)^{1/2}g(x)=g(x),
$$ 
we get the desired result.
\end{proof}

\begin{remark}
The case $\alpha=1$ is a special one. As we can see explicitly in the radially
symmetric situation, it is possible to have no solution or an infinite
number of solutions. Indeed, let $g(x)=\aaa|x|$, as it is easily proved by
Schwarz symmetrization (see section \ref{symmetrysec}), the solution
has to be a ball. Now, looking for a ball $B_R$
of radius $R$ solving (\ref{ovfbp1}) is equivalent to solve $g(R)=R/N$ (because
$u_{B_R}=(R^2-|x|^2)/2N$) and the result follows according to the value of $\aaa$.
\end{remark}

\subsection{Proof of Theorem \ref{mainthm}}\label{secmainthm}

The proof splits into two parts which are separated in four paragraphs.
In the first part we prove the existence (and boundedness) of a solution, while in the latter the proof of regularity is presented.

More precisely in Paragraph \ref{proofex} we follow the lines of \cite{Crou}, cf also \cite{HP} and we use a concentration-compactness argument as in \cite{buc} to prove the existence of a minimizer which  is a \emph{quasi-open} set and which may be unbounded. 
We refer to \cite{HP} for a precise definition and discussion of the concept of quasi-open set; let us only remind that if $u\in H^1(\rn)$, then its super level sets $\{u>t\}$ are quasi-open, for each $t\in\R$.

In Paragraph \ref{proofbd}, using the notion of {\it local
shape subsolution} introduced in \cite{buc2}, we prove that the minimizer is
in fact bounded.
In Paragraphs \ref{proofreg1}, \ref{proofreg2} we prove the regularity of the minimizer as in \cite{Bri} (see also \cite{BHP}).
The main difficulty is to prove that it is actually an open set, then we can conclude to higher
regularity by classical techniques from free boundary problems like in \cite{AlCa} and \cite{GuSh}.

\subsubsection{Proof of existence} \label{proofex}

Let us introduce the following auxiliary problem
\begin{eqnarray} \label{aux}
\min\{G(v);\,v\in H^1(\R^N),\; \phi(\Om_v)\le 1\},
\end{eqnarray}
where $\Om_v$ denotes the quasi-open set $\Om_v:=\{x\in \R^N;v(x)> 0\}$, $\phi$ is defined in (\ref{defphi}) and $G$ 
is the functional defined by
\begin{equation}
G(v)=\frac{1}{2}\int_{\R^N} |\nabla v|^2\,dx - \int_{\R^N} v\,dx\,.
\end{equation}
Let us prove the existence of a minimizer $u$ for problem (\ref{aux}).

We use the classical Poincar\'e inequality, valid for any set
of bounded measure (see Lemma 4.5.3 in \cite{HP}).
As already noticed, the
constraint $\phi({\Om_v}) \leq 1$ implies that the measure of 
$\Om_v$ is uniformly bounded, that is $\exists m>0, |\Om_v|\leq m$. 
The Poincar\'e inequality writes as follows: there exists $C=C(N)>0$ such that for every $v\in H^1(\R^N)$ satisfying 
$|\Omega_v|\leq m$ it holds
\begin{equation}\label{poinc}
\int_{\R^N} v^2\leq C\,m^{\frac 2N}\int_{\R^N} |\nabla v|^2.
\end{equation}
Therefore, since $\int_{\Omega_v} v \leq m^{1/2}\left(\int_{\R^N} v^2\right)^{1/2}$, we have
$$
2G(v)\geq \int_{\R^N}|\nabla v|^2-C'\|v\|_{H^1(\R^N)},
$$
and $G(v)$ is estimated from below and a minimizing sequence $u_n$ is
necessarily bounded in $H^1(\R^N)$. 

Now we use a concentration compactness argument for the quasi-open sets $A_n=\{u_n>0\}$. 
Following \cite[Theorem 2.2]{buc} two situations may occur:
\begin{description}
 \item[Dichotomy] The sequence $\{A_n\}$ splits into two parts: $A_n=A_n^1\cup A_n^2$ with $d(A_n^1,A_n^2)\to
+\infty$ and $\liminf |A_n^i|>0$. The \emph{resolvent operators} satisfy $\|R_{A_n} - R_{A_n^1\cup A_n^2}\|\to 0$
in the operator norm (see \cite{buc} for details on the resolvent operator).
\item[Compactness] There exists a sequence of vectors $y_n\in \R^N$ and a positive Borel measure $\mu$
(vanishing on sets of zero capacity) such that $y_n+A_n$ $\gamma$-converges to the measure $\mu$
(and the resolvent operators satisfy $\|R_{y_n+A_n} - R_{\mu}\|\to 0$
in the operator norm).
 \end{description}

Notice that in the situation of Problem (\ref{aux}), dichotomy cannot occur because since $g^2\to +\infty$ at infinity, the constraint
$\int_{A_n} g^2 \leq 1$ prevents a subpart of $A_n$ of measure bounded from below to move to infinity.
Thus, we are in the compactness situation and we denote by $A_\mu$ the regular set of the limit measure $\mu$, defined
as
$$
A_\mu:=\{\bigcup A: A \mbox{ is finely open, } \mu(A)<\infty\}.
$$
Then the sequence $v_n(x)=u_n(x-y_n)=R_{y_n+A_n}(1)$
converges to $v=R_\mu(1)\in H^1_0(A_\mu)$  weakly in $H^1$ and almost everywhere
(and $A_\mu=\{v>0\}=\Om_v$).
Notice that at this point we can not say that $\Omega_v$ provides a solution to minimization problem \eqref{aux}, since the constraint is not translation invariant. 
We can avoid this problem arguing for istance as in \cite{BV}, proving that the sequence $\{y_n\}$ is bounded, thanks again to the behaviour of $g$ at infinity. 
Indeed first choose $R>0$ such that $\int_{B(O,R)}v^2\,dx=\alpha>0$ (here and in the sequel we denote by $B(x,R)$ the ball centered at $x$ with radius $R$); by the convergence of $v_n$ to $v$, we have
\begin{equation}\label{intun2}
\int_{B(0,R)}v_n^2dx=\int_{B(-y_n,R)}u_n^2dx=\int_{B(-y_n,R)\cap A_n}u_n^2dx\geq\frac{\alpha}{2}
\end{equation}
for $n$ large enough. Since $|A_n|$ is bounded, $\|u_n\|_\infty$ is bounded, say $\|u_n\|\leq M$; then \eqref{intun2} implies
$$
|A_n\cap B(-y_n,R)|\geq \frac{\alpha}{2M^2}\,,
$$
whence
$$
1\geq\int_{A_n}g^2dx\geq\int_{A_n\cap B(-y_n,R)}g^2dx\geq\frac{\alpha}{2M^2}\inf_{B(-y_n,R)}g^2\,.
$$
The latter leads to a contradiction if we assume that $\|y_n\|$ is unbounded, since we would have $\inf_{B(-y_n,R)}g^2\to\infty$.
So far, we have proved $y_n$ is bounded. 
Then it converges to some $y_0$ (up to a subsequence) and 
the sequence $u_n$
converges to $u(x)=v(x+y_0)$  weakly in $H^1$ and almost everywhere.
We set  $\Omega^*=\{u>0\}=A_\mu-y_0$ and
by Fatou
Lemma, we infer that 
$$
\int \chi_{\Om^*} g^2 dx \leq \liminf \int
\chi_{\Om_{u_n}} g^2 dx \leq 1
$$ 
and so the constraint is satisfied.
We deduce that $\Om^*$
provides a solution of the shape optimization problem
(\ref{s1}), (\ref{s2}), as it is classical in situations where the objective function
$J$ is monotone decreasing with respect to set inclusion.

\subsubsection{Proof of boundedness}\label{subbound} \label{proofbd}
We recall the definition of local shape subsolution introduced by D. Bucur in \cite{buc2}:
a set $A$ is a \emph{local shape subsolution} for the energy problem if there exist $\delta>0$ and $\Lambda>0$ such that for any quasi-open set $\tilde{A}\subset A$ with 
$\|u_{\tilde{A}}-u_A\|_{L^2}<\delta$ we have
\begin{equation}\label{shapesub}
J(A)+\Lambda |A|\leq J(\tilde{A})+\Lambda |\tilde{A}|.
\end{equation}
In \cite[Theorem 2.2]{buc2}, it is proved that any local shape subsolution is bounded (and has finite perimeter). 
Thus our aim is to prove that $\Om^*$ is a local shape subsolution. We argue by contradiction: let us assume that there exists
a sequence $\lambda_n$ going to 0 and a sequence $\Omega_n\subset \Omega^*$ such that
\begin{equation}\label{bd1}
J(\Omega_n)+\lambda_n|\Omega_n|< J(\Omega^*)+\lambda_n |\Omega^*|.
\end{equation} 
We can assume that $\Omega_n$ is an increasing sequence converging to $\Omega^*$ in the $L^2(\rn)$-norm of $u_{\Omega_n}$, that is $\|u_n-u\|_{L^2(\rn)}$ tends to zero, where $u=u_{\Omega^*}$. 
Indeed, if $\Omega_n$ converges to a strictly smaller set, then equation (\ref{bd1}) cannot hold by monotonicity of the energy $J$ and this would give a contradiction.
Fix $t_n>1$ such that 
$$
\phi(t_n \Omega_n)=t_n^{2\alpha +N}\int_{\Omega_n}g^2=1.
$$
Then necessarily $t_n\to 1$. By minimality of $\Omega^*$, 
$J(t_n\Omega_n)=t_n^{N+2}J(\Omega_n)\geq J(\Omega^*)$. \\
Plugging into (\ref{bd1}) yields
\begin{equation}\label{bd2}
 J(\Omega^*)\left\lbrack\frac{1-t_n^{N+2}}{t_n^{N+2}}\right\rbrack \leq \lambda_n(|\Omega^*|-|\Omega_n|).
\end{equation} 
We divide both sides of (\ref{bd2}) by $t_n^{N+2\alpha}-1=\int_{\Omega^*}g^2\left(\int_{\Omega_n}g^2\right)^{-1}-1$
and we get
\begin{equation}\label{bd3}
 \frac{t_n^{N+2}-1}{t_n^{N+2\alpha}-1}\left(-\frac{J(\Omega^*)}{t_n^{N+2}}\right) \leq 
\frac{\lambda_n|\Omega^*\setminus\Omega_n|\int_{\Omega_n}g^2}{\int_{\Omega^*\setminus\Omega_n}g^2}.
\end{equation} 
Now obviously, we just need to prove boundedness far from 0, thus we can assume that we are outside
a fixed ball $B(0,R)$ (this notion of local shape subsolution can be localized).
Therefore, 
$$
\dfrac{\int_{\Omega^*\setminus\Omega_n}g^2}{|\Omega^*\setminus\Omega_n|}\geq \min_{\R^N\setminus B(0,R)} g^2>0.
$$
We pass to the limit $\lambda_n\to 0$ and $t_n\to 1$; since the left-hand side of (\ref{bd3}) converges to $\frac{N+2}{N+2\alpha}\,(-J(\Omega^*)>0$ and the right-hand side tends to 0, we reach the desired contradiction.

\subsubsection{Proof of regularity} \label{proofreg2}

The proof of regularity of the optimal shape is
very similar to the proof given in \cite{Bri} (see also \cite{BHP})
where the constraint is $|\Omega_u|\leq 1$ instead of
$\int_{\Omega_u}g^2\leq 1$. 
These papers are themselves inspired by
\cite{AlCa} and \cite{GuSh}. Thus, in the
sequel, we will mainly emphasize the particularities of our
situation.

Before giving the details, we need to show some preliminary results.

Let us denote by $u$ the minimizer of problem (\ref{aux}). Since we know that
the minimizer is bounded, let us denote by $D$ a fixed ball containing $\Omega_u$.
The first
step is to prove that $u$ is continuous in $D$ (and therefore the
set $\Om_{u}=\{x\ :\ u(x)>0\}$ is open). To get rid of the constraint
(in order to be able to test with a wider class of functions), we first prove
that the minimization problem (\ref{aux}) is equivalent to a penalized problem.
\begin{lemma}\label{lemma1reg}
There exists $\lambda >0$ such that for any $v\in H^1_0(D)$
\begin{equation}\label{pen1}
G(u)\leq G(v)+\lambda\left(\int_{\Om_v} g^2(x)\,dx -1\right)^+\,.
\end{equation}
\end{lemma}

It is remarkable that the two problems are equivalent, not only when $\la$
goes to infinity as usual, but for a finite value of $\la$.
\begin{proof}[Proof of Lemma \ref{lemma1reg}] 
For a fixed $\la>0$,
let us denote by $G_\la$ the functional
$$
G_\la(v):=G(v)+\la\left(\phi({\Om_v}) -1\right)^+. 
$$
The existence of a minimizer $u_\la$ for the problem 
$$
\inf_{v\in H^1_0(D)} G_\la(v),
$$ 
is obtained in an analogous way as the
existence of a minimizer for problem (\ref{aux}) above. 
If $\phi({\Om_{u_\la}}) \leq 1$, we get
$G_\la(u_\la)=G(u_\la)$ and since $u_\la$ and $u$ are both
minimizers of $G_\la$ and $G$, the result follows.  
It remains to prove that we cannot have
\begin{equation}\label{reg2}
\phi({\Om_{u_\la}}) > 1,
\end{equation}
for $\la$ large enough. 
Assume, by contradiction, that it is the case and let
us introduce $u^t=(u_\la - t)^+$. 
Differentiating with respect to $t$ and using the
co-area formula, we get
\begin{eqnarray*}
\frac{d}{dt}{\Big|_{t=0}}\int_{\{u_\la>t\}} |\nabla u_\la|^2\,dx &=& -\int_{\{u_\la=t\}} |\nabla u_\la|\,d\mathcal{H}^{N-1},\\
\frac{d}{dt}{\Big|_{t=0}}\int_{\{u_\la>t\}} g^2\,dx &=& - \int_{\{u_\la=t\}} \frac{g^2}{|\nabla u_\la|}\,d\mathcal{H}^{N-1},
\end{eqnarray*}
while
$$
\frac{d}{dt}{\Big|_{t=0}} \int_{\{u_\la>t\}} (u_\la - t)\,dx=- \int_{\{u_\la>t\}} \,dx\,.
$$ 
Thus $\frac{d}{dt}{\Big|_{t=0}} G_\la(u^t)\geq 0$ yields
$$
\int_{\{u_\la=t\}} \frac{1}{2} |\nabla u_\la| + \la \frac{g^2}{|\nabla u_\la|}\,d\mathcal{H}^{N-1}\leq \int_{\{u_\la>t\}} \,dx
\leq |D|.
$$
Now
\begin{equation}\label{20a}
\int_{\{u_\la=t\}} \frac{1}{2} |\nabla u_\la| + \la \frac{g^2}{|\nabla u_\la|}\,d\mathcal{H}^{N-1}\geq \sqrt{2\la}
\int_{\{u_\la=t\}} g\,d\mathcal{H}^{N-1},
\end{equation}
therefore, if we can prove that $\int_{u_\la=t} g\,d\mathcal{H}^{N-1}$
is estimated from below by a positive constant, (\ref{20a}) would lead to the
desired contradiction for $\la$ large enough. 
This is the content of the following lemma.
\end{proof}

\begin{lemma}\label{lemweightisop}
There exists a positive constant $C$ such that for any measurable set
$\omega\subset D$ with $\int_\omega g^2\,dx\geq 1$, it holds
\begin{equation}\label{isop1}
\int_{\partial\omega} g\,d\mathcal{H}^{N-1}\geq C.
\end{equation}
\end{lemma}
\begin{proof}
Let us assume, by contradiction, that there exists a sequence $\omega_n$ such that
$$
\int_{\omega_n} g^2\,dx\geq 1\quad\mbox{and}\quad
\int_{\partial\omega_n} g\,d\mathcal{H}^{N-1}\leq \frac{1}{n}\,.
$$
For any $R$, let $B_R$ be the ball centered at $O$ with radius
$R$ and 
$$
g_R=\min\{g(x)\ :\ x\in\rn, |x|=R\}.
$$ 
We fix $R>0$ such that $\int_{B_R} g^2\,dx<1/2$.
We have
$$
|\partial\omega_n\setminus B_R| \; g_R \leq \int_{\partial\omega_n\setminus B_R} g\,d\mathcal{H}^{N-1}\leq \frac{1}{n}.
$$
Thus $|\partial\omega_n\setminus B_R|\leq \frac 1{ng_R}$. 
By the Relative Isoperimetric Inequality on $\R^N\setminus B_R$, see for instance \cite{CGR}, there exists a positive constant
$c_0$ such that
$$
|\omega_n\setminus B_R|\leq c_0 |\partial\omega_n\setminus B_R|^{\frac N{N-1}}.
$$
This implies that $|\omega_n\setminus B_R|$ can be chosen smaller than $1/2M$ where $M=\max_D g^2$, since 
$|\partial\omega_n\setminus B_R|$ tends to zero. 
Hence
$$
1\leq \int_{\omega_n} g^2\,dx=\int_{\omega_n\setminus B_R} g^2\,dx+\int_{\omega_n\cap B_R} g^2\,dx
< \frac{1}{2}+\frac{1}{2},
$$
which is a contradiction.
\end{proof}

\begin{remark}
By homogeneity, statement (\ref{isop1}) of Lemma \ref{lemweightisop} can also be written as:
\begin{equation}\label{isop2}
\exists C>0, \,\mbox{such that}\; \forall \omega\subset D,\qquad 
\int_{\partial\omega} g\,d\mathcal{H}^{N-1}\geq C \left(\int_\omega g^2\,dx\right)^{\frac{\alpha+N-1}{2\alpha+N}}\,.
\end{equation}

Notice that (\ref{isop2}) is a kind of weighted isoperimetric inequality and has its own importance: an interesting question would be to determine its optimal domains. 
Using the theory of quasi-minimizers, one should be able to prove that such an
optimal domain $\omega^*$ is regular. Now, the differentiation with
respect to the domain would give as first order optimality
condition:
$$
Hg+\frac{\p g}{\p n}=\gamma g^2 \quad\mbox{on}\ \p\omega^*,
$$
where $H$ is the mean curvature of the boundary of $\omega^*$ and $\gamma$ a Lagrange multiplier.
\end{remark}

\subsubsection{Details of the proof of regularity} \label{proofreg1}
Let us fix a ball $B(x_0,r)$
and, as a test function in Lemma \ref{lemma1reg}, let us choose $v$
defined as $v=u$ in $\rn\setminus\overline{B(x_0,r)}$ and $v$ solution of
$$
\begin{cases}
-\Delta v=1 & \qquad\text{ in }\; B(x_0,r)\\
v=u & \qquad\text{ on }\;\p B(x_0,r),
\end{cases}
$$ 
inside the ball. By the maximum principle $v>0$ in $B(x_0,r)$ and, therefore, 
$$
\Om_v=\Om_u\cup B(x_0,r),
$$ 
where we recall that $\Om_v=\{x\ :\ v(x)>0\}$.
Thus, since
$\int_{\Om_u} g^2(x)\,dx =1$, we have 
$$
0\leq \int_{\Om_v} g^2(x)\,dx -1 \leq \int_{B(x_0,r)} g^2 \leq C r^N.
$$ 
It follows, from Lemma \ref{lemma1reg}, that $u$ satisfies
$$
\int_{B(x_0,r)}|\nabla(u-v)|^2 \leq \lambda C r^N
$$
and, by classical regularity results (see e.g. Theorem 3.5.2 in \cite{Mo}), $u$ is H\"older continuous on $D$. 
Consequently the set $\Omega_u$ is open as claimed at the beginning of the proof. 
In particular, following \cite{BHP}, we can prove that $u$ is Lipschitz on $D$. 

Let us now study $\Delta u +\chi_{\Omega_u}$. 
The fact that $\Delta u+1=0$ on $\Omega_u$ (in the sense of distributions) is easily obtained using perturbations of the kind $v=u+t\varphi$ with $\varphi\in C_0^\infty(\Omega_u)$.
Then, following step by step \cite[Theorem 2.2,Proposition 2.3]{Bri}, one can prove that $\Delta u +\chi_{\Omega_u}=\mu$, where $\mu$ is a (positive) Radon measure, supported by $\p\Omega_u$ and absolutely continuous with respect to the Hausdorff measure $\mathcal{H}^{N-1}$ in $D$. 
Then, using a blow-up technique near the boundary points of $\p\Omega_u$, we prove more precisely like in \cite[Theorem
5.1]{Bri} that
$$
\Delta u +\chi_{\Omega_u}=g \mathcal{H}^{N-1}\lfloor \p\Omega_u.
$$
We need for that purpose that $g$ is estimated from below: $g(x)\geq
c>0$ which is true as soon as we are far from the origin. We can conclude using \cite[Theorem 2.13, Theorem 2.17]{GuSh}, at least
outside the origin where $g=0$. 
Notice that in Proposition
\ref{propO} below it is proved that it is certainly the case when $\alpha>1$.

\section{Basic Properties}\label{basicpropSec}
In this section, we prove some basic properties of the solution to Problem (\ref{ovfbp1}) or (\ref{min1}). 

The first one is a somewhat technical property which will be necessary in many cases in the sequel. 
It states that the origin $O$ is inside the domain (at least when $\alpha>1$) and it is not surprising of course. Indeed, since $g$ is increasing with respect to $|x|$, one can easily imagine that translating a domain $\Omega$ toward the origin should make $\phi(\Omega)$ smaller while the torsion remains unchanged. Unfortunately we can prove such a property only for $\alpha>1$, as most of our next results.

\begin{proposition}\label{propO}
Let $\Omega$ be a solution of the minimization problem (\ref{min1})
and assume that the homogeneity degree of $g$ satisfies
$\alpha>1$. Then the origin $O$ belongs to $\Omega$.
\end{proposition}
\begin{proof} 
Let us begin by proving that $O$ cannot be in $\rn\setminus\overline{\Omega}$.
Indeed in this case, there would exist a ball $B_\ep=B(O,\varepsilon)$ of small radius $\varepsilon$ and center at the origin,
such that $B(O,\ep)\subset\rn$. 
Let $\Omega_\varepsilon=\Omega\cup B_\ep$.
We have immediately
$$
\phi(\Omega_\ep)=\phi(\Omega)+\ep^{2\alpha+N}\phi(B_1),
$$
where $B_1$ is the unit ball centered at $O$. In the same way,
$$
J(\Omega_\ep)=J(\Omega)+\ep^{N+2}J(B_1).
$$
Therefore, since $\alpha>1$ we get the following expansion
$$
\phi(\Omega)^{-\frac{2+N}{2\alpha+N}} J(\Omega_\ep)=
\phi(\Omega)^{-\frac{2+N}{2\alpha+N}} J(\Omega)\left[1+\ep^{N+2}
\frac{J(B_1)}{J(\Omega)} + o(\ep^{N+2})\right],
$$
which contradicts the optimality of $\Om$ since, as $J(B_1)<0$, it implies $J(\Om_\ep)<J(\Om)$. 
Continuing to argue by contradiction, let us now assume that $O\in\Omega$. 
If $\partial\Omega$ satisfies an exterior cone condition at $O$, that is if there exists a cone $C_\ep=\ep C_1$ of vertex
$O$ and size $\ep$ such that $C_\ep\cap{\Omega}=\emptyset$, we can do exactly the
same construction as before considering $\Omega_\ep=\Omega\cup C_\ep$ and we
conclude in the same way. 

So far we are reduced to consider $O\in\partial\Omega$ with no exterior cone condition at $O$. 
In this case we use an Harnack inequality as in \cite[Corollary 2.6]{GuSh}. 
Let us denote by $\delta(x)$
the distance from $x$ to $\Om^c$. For $x\in\Om$ close to $O$ we have
\begin{equation}\label{bgp1}
    |\nabla u_\Om(x)|\leq N2^N\left(\sup_{B(x,2\delta(x))} g + \delta(x)\right).
\end{equation}
The latter shows that $|\nabla u_\Om(x)|$ has to go to zero when $x$ approaches $O$.
This leads to a contradiction (if $O\in\partial\Omega$ with no
exterior cone condition at $O$), since  by a simple barrier
argument (comparison between $u_\Om$ and the stress function of a cone
of size greater than $\pi$) we would have $\lim\sup_{x\to 0} |\nabla u_\Om(x)|=
+\infty$.
\end{proof}

Next we prove the monotonicity of the solutions of \eqref{ovfbp1} with respect to $g$.

\begin{theorem}\label{teomonot}
Let $g_1,g_2$ satisfy (\ref{assum_g}) with $\alpha>1$ and let $\Omega_1, \Omega_2$ be solutions to Problem (\ref{ovfbp1}) related to $g_1$ and $g_2$, respectively. 
Assume $g_1(x)\ge g_2(x)$ for every $x\in\R$, then $\Omega_1\subseteq \Omega_2$.
\end{theorem}
\begin{proof}
Assume by contradiction $\Omega_1\not\subset\Omega_2$ and consider $t\Omega_1$ with
$$
t=\sup\{s>0 \ :\ s\Omega_1\subseteq\Omega_2\}.
$$
Then $0<t<1$ (notice that $t>0$ comes from Proposition \ref{propO}) and
there exists $\xx\in\bd t\Omega_1\cup\bd\Omega_2$, with $\nu_{t\Omega_1}(\xx) || \nu_{\Omega_2}(\xx) || \nu$, where $\nu_\Omega(x)$ denotes the outer unit normal to $\bd\Omega$ at $x$ and $\nu\in S^{N-1}$.
\newcommand{\Gr}{
(-6,0) to[out=90,in=200] (0,5) to[out=20,in=180] (5,6) to[out=0,in=90] (8,0) to[out=270,in=0] (0,-5) to[out=180,in=315] (-4,-4) to[out=135,in=270] (-6,0)
}
\newcommand{\disegnoomega}{
(-3.5,0) to[out=90,in=200] (0,5) to[out=10,in=180] (2,5.2) to[out=0,in=90] (4,0) to[out=270,in=0] (0,-3) to[out=180,in=270] (-3.5,0)
}
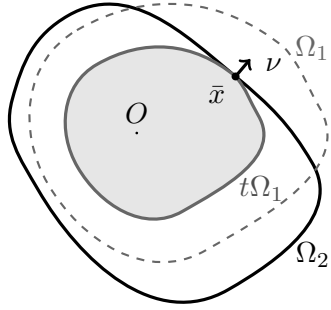
\begin{figure}[h]
\centering
\begin{tikzpicture}[x=3mm,y=3mm]
\begin{scope}[rotate=300]
\draw[very thick]\Gr;
\draw (8.5,4) node{$\Omega_2$};
\draw[very thick, gray!120, fill=gray!20] (-3.5,0) to[out=90,in=200] (0,5) to[out=10,in=180] (2,5) to[out=0,in=90] (4,0) to[out=270,in=0] (0,-3) 
to[out=180,in=270] (-3.5,0);
\draw (4,4) node[below, gray!140]{$t\Omega_1$};
\fill[black] (0,0) circle (0.5pt);\draw (0,0)node[above]{$O$};
\draw[->, very thick] (90:5)-- ++(110:1);
\draw (0,6) node[right]{$\nu$};
\fill (90:5) circle (1.5pt);
\draw (0,5) node[below left]{$\xx$};
\begin{scope}[xscale=1.5, yscale=1.5]
\draw[dashed, thick, gray!120] (-3.5,0) to[out=90,in=200] (0,5) to[out=10,in=180] (2,5) to[out=0,in=90] (4,0) to[out=270,in=0] (0,-3) 
to[out=180,in=270] (-3.5,0);
\draw (90:5) node[right, gray!120]{$\Omega_1$};
\end{scope}
\end{scope}
\end{tikzpicture}
\caption{$t\Omega_1\subseteq\Omega_2$ with $\xx\in\bd(t\Omega_1)\cap\bd\Omega_2$}\label{Fig-proofunicita}
\end{figure}

We want to compare $u_{t\Omega_1}$ and $u_{\Omega_2}$, the stress functions of $t\Omega_1$ and $\Omega_2$, respectively.
Notice that $u_{t\Omega_1}(x)=t^2 u_{\Omega_1}(\frac xt)$.
Define $w=u_{\Omega_2}-u_{t\Omega_1}$; it satisfies
$$
\begin{cases}
\Delta w = 0 &\qquad\text{ in }t\Omega_1,\\
w\ge 0 &\qquad\text{ on }\bd t\Omega_1,\\
w(\xx)=0.
\end{cases}
$$
Hence by Hopf Lemma, it holds $\frac{\partial w}{\partial \nu}(\xx)>0$.
Notice that
$$
\frac{\partial w}{\partial \nu}(\xx) = |\nabla u_{\Omega_2}(\xx)|-|\nabla u_{t\Omega_1} (\xx)|= g_2(\xx)-tg_1(\dfrac {\xx}t),
$$
since $\nu$ is parallel to $\nabla u_{\Omega_2}(\xx), \nabla u_{t\Omega_1} (\xx)$ and $|\nabla u_{t\Omega_1} (\xx)|=t|\nabla u_{\Omega_1}(\frac {\xx}t)|$, with $\frac{\xx}t\in\bd\Omega_1$.
Hence, by the homogeneity of $g_1$, and the fact that $t<1, \alpha>1$, we get
$$
g_2(\xx)> tg_1(\frac{\xx}t)=t^{1-\alpha}g_1(\xx) > g_1(\xx),
$$
which contradicts the assumption $g_1\ge g_2$.
\end{proof}
\begin{remark} We can prove the above theorem under some weaker assumptions, precisely it is sufficient that at least one between $g_1$ and $g_2$ satisfies (\ref{assum_g}), with $\alpha>1$. In such a case we have however to assume $O$ belongs to the interior of both $\Omega_1$ and $\Omega_2$ and that they are bounded.
In case $g_1$ is the function satisfying assumption (\ref{assum_g}), the proof is precisely the same as above. Notice that in this case $O\in\Omega_1$ and $\Omega_1$ bounded are pleonastic assumptions since they implied by Proposition \ref{propO} and he result os Section 2, respectively.
If $g_2$ satisfies (\ref{assum_g}) with $\alpha>1$ instead of $g_1$, then we argue again by contradiction, very similarly as before (in this case $O\in\Omega_2$ and $\Omega_2$ bounded are pleonastic). Indeed 
if $\Omega_1\not\subset\Omega_2$, consider $t\Omega_2$ with
$$
t=\inf\{s>0 \ :\ \Omega_1\subseteq s\Omega_2\}.
$$
Then $t>1$ and
there exists $\xx\in\bd \Omega_1\cup\bd (t\Omega_2)$, with $\nu_{\Omega_1}(\xx) || \nu_{t\Omega_2}(\xx) || \nu$, where $\nu_\Omega(x)$ denotes the outer unit normal to $\bd\Omega$ at $x$ and $\nu\in S^{N-1}$. Then we can proceed in the same way as before.
\end{remark}

As a natural straightforward corollary of the previous theorem, the uniqueness of the solution follows.

\begin{theorem}\label{propunique}
If $g$ satisfies assumptions (\ref{assum_g}) for $\alpha>1$, then the solution is unique.
\end{theorem}

\section{Starshape and Connectedness}\label{starshapeSec}
We recall that a set $\Omega$ is said {\em starshaped with respect to a point $x_0\in\Omega$} if
$$t(x-x_0)+x_0\in\Omega\quad \text{for every $x\in\Omega$ and every $t\in[0,1]$}.
$$
When $x_0=O$ we simply say that $\Omega$ is starshaped.

\begin{theorem}\label{propstarshape}
If $g$ satisfies ssumptions (\ref{assum_g}) for $\alpha>1$, then $\Omega$ is starshaped (with respect to $O$).
\end{theorem}
\begin{proof}
The proof is similar to that one of Theorem \ref{teomonot}.

By contradiction, assume $\Omega$ is not starshaped and let
$$
t=\inf\{s\in[0,1]\,:\, s\Omega\not\subset\Omega\}.
$$
There exists $\xx\in\partial(t\Omega)\cap\partial\Omega$.
Notice that $0<t<1$ since the origin belongs to $\Om$, and $\Omega$ is bounded and not starshaped.
{In a similar way as before}, consider the function
$$
u_t(x)=t^2u_{\Omega}\left(\frac{x}{t}\right)\,.
$$
Then $u_t$ is the stress function of $t\Omega$ and whence
$$
\begin{cases}
-\Delta w=0 & \qquad\text{in }\;t\Om\\
w\geq0 & \qquad\text{on }\;\p(t\Om)\\
w(\xx)=0,
\end{cases}
$$
where $w=u_\Omega-u_t$.
\newcommand{\disegnoomega}{
(-3.5,0) to[out=90,in=200] (0,5) to[out=10,in=180] (2,5.2) to[out=0,in=90] (4,0) to[out=270,in=0] (0,-3) to[out=180,in=270] (-3.5,0)
}
\begin{figure}[h]
\centering
\begin{tikzpicture}[x=5mm,y=5mm]
\begin{scope}[rotate=20]
\draw[very thick] \disegnoomega;
\draw (4,4) node{$\Omega$};
\draw[very thick] (1.5,3.9) circle (0.5);
\begin{scope}[scale=0.65]
\draw[very thick, dashed] \disegnoomega;
\fill[gray!10] \disegnoomega;
\draw (3,0) node{$t\Omega$};
\draw[very thick, dashed] (1.5,3.9) circle (0.5);
\fill[white] (1.5,3.9) circle (0.5);
\end{scope}
\fill[black] (1.45,3.35)node[below]{$\xx$} circle (0.7pt);
\draw[->, thick, gray] (1.45,3.35)-- ++(0,0.6) node[right]{$\nu$};
\fill[black] (0,0) circle (0.5pt);
\draw (0,0)node[above]{$O$};
\end{scope}
\end{tikzpicture}
\caption{$t=\inf\{ s\in[0,1]: \ s\Omega\nsubseteq\Omega\}$.}\label{Fig-proofstarshpe}
\end{figure}

Then
\begin{equation}\label{hopfstarshape}
\frac{\partial w}{\partial\nu}(\xx)\leq0\,,
\end{equation}
where $\nu$ is the outer unit normal at $\xx$ of $\Omega$ and $t\Omega$.
On the other hand
$$
\frac{\partial w}{\partial\nu}(\xx)=|\nabla u_t(\xx)|-|\nabla u_\Omega(\xx)|\,,
$$
then \eqref{hopfstarshape} reads as
\begin{equation}\label{2becontradictedstarshape}
|\nabla u_t(\xx)|\leq |\nabla u(\xx)|=g(\xx)\,.
\end{equation}
Moreover a straightforward calculation gives
$$
\nabla u_t(x)=t\nabla u_{\Omega}(\frac{x}{t})\quad\text{for }x\in t\Omega\,,
$$
which entails
$$
|\nabla u_t(\xx)|=t|\nabla u_{\Omega}(\xx/t)|=tg(\xx/t)=t^{1-\alpha}g(\xx)>g(\xx),
$$
for $1-\alpha<0$, $t<1$ and $\xx\neq 0$ (then $g(\xx)>0$).

The latter contradicts \eqref{2becontradictedstarshape} and the proof of starshapedness is concluded.
\end{proof}

Starshape obviously
implies connectedness. Then when $\alpha>1$ the solution is connected. We are able to
prove this property also when $\alpha<1$.

\begin{proposition}
Let $\alpha\neq 1$and $\Omega$ be a solution of the minimization problem (\ref{min1}) or (\ref{min2}). Then $\Om$ is connected.
\end{proposition}
\begin{proof}
For $\alpha>1$ the thesis follows from Theorem \ref{propstarshape} as already pointed out. 
It is then sufficient to consider the case $\alpha < 1$.
Let us assume, by contradiction, that the solution $\Om$ is not connected
and let us write it as
$$
\Om=\Om_1\cup\Om_2,\  \text{ with }\ \Om_1\cap\Om_2=\emptyset,\ |\Om_1|>0,\ |\Om_2|>0.
$$
Since $g$ is positive outside 0, we have $\phi(\Om_i)>0$ for $i=1,2$. Let us
denote by $M$ the value of the minimum
\begin{equation}\label{bgp2}
    M=\phi(\Omega)^{-\frac{2+N}{2\alpha+N}} J(\Omega)<0.
\end{equation}
For each component, we have
$\phi(\Omega_i)^{-\frac{2+N}{2\alpha+N}} J(\Omega_i)\geq M$, $i=1,2$.
Now, $J(\Om)=J(\Om_1)+J(\Om_2)$ and $\phi(\Om)=\phi(\Om_1)+\phi(\Om_2)$.
Therefore
$$
J(\Om)\geq M\left(\phi(\Omega_1)^{\frac{2+N}{2\alpha+N}}+
\phi(\Omega_2)^{\frac{2+N}{2\alpha+N}}\right)>M\left(\phi(\Omega_1)+
\phi(\Omega_2)\right)^{\frac{N+2}{2\alpha+N}}=M\phi(\Omega)^{\frac{N+2}{2\alpha+N}},
$$
the strict inequality coming from the fact that $\alpha<1$ (whence $\frac{N+2}{2\alpha+N}>1$), $\phi(\Omega_i)>0$ for $i=1,2$ and $M<0$. This
clearly leads to a contradiction with (\ref{bgp2}).
\end{proof}

\section{convexity}\label{convexitySec}
Let us recall some definitions which will be usefull later on.
A lower semicontinuous function $u:\rn\to\R\cup\{\pm\infty\}$ is said \emph{quasi-convex} if it has convex sublevel sets, or, equivalently, if
$$
u\left( (1-\la)x_0+\la x_1 \right)\le \max\{ u(x_0),u(x_1) \},
$$
for every $\la\in[0,1]$, and every $x_0, x_1\in\rn$.
If $u$ is defined only in a proper subset $\Omega$ of $\R^n$, we extend $u$ as $+\infty$ in $\R^n\setminus\Omega$ and we say that $u$ is quasi-convex in $\Omega$ if such an extension is quasi-convex in $\rn$.
In an analogous way, $u$ is \emph{quasi-concave} if $-u$ is quasi-convex, i.e. if it has convex superlevel sets.
Obviously, if $u$ (or any one of its powers) is convex (concave) then it is quasi-convex (quasi-concave) but the reverse is not necessarily true.

\begin{remark}\label{rmkalphaconvexity}
Notice that, due to the $\alpha$-homogeneity, the quasi-convexity of $g$ is equivalent (for $\alpha>0$) to the following 
apparently stronger property:
$$
g^{1/\alpha}\,\text{ is convex.}
$$
\end{remark}
Indeed, notice that $h(x)=g^{\frac 1{\alpha}}(x)$ is homogeneous of degree one and it is quasi-convex.
Fix $x_0,x_1\in\rn$ and consider $\la\in [0,1]$.
We want to prove that
$$
\frac{h((1-\la)x_0+\la x_1)}{(1-\la)h(x_0)+\la h(x_1)}\le 1.
$$
Denote by 
$$
\xi=\frac{\la h(x_1)}{(1-\la) h(x_0)+\la h(x_1)};
$$
using the quasi-convexity and the homogeneity of $h$, we get
$$
\frac{h((1-\la)x_0+\la x_1)}{(1-\la)h(x_0)+\la h(x_1)} = h\big((1-\xi)\frac{x_0}{h(x_0)} + \xi \frac{x_1}{h(x_1)}\big) \le \max\Big\{ h(\frac{x_0}{h(x_0)}), h(\frac{x_1}{h(x_1)}) \Big\} = 1.
$$
\begin{theorem}\label{propconvexity}
Let $g$ be a quasi-convex and homogeneous function of degree $\alpha\geq 2$, with $g(x)>0$ for $x\neq 0$.
Then $\Omega$ is convex.
\end{theorem}
\begin{proof}
The proof follows the same lines of that of starshape. 
By contradiction, assume $\Omega$ is not convex and let $\Omega^*$ be its convex hull.
Denote by $u_\Omega$ the solution of \eqref{i1}, by $u^*$ its concave envelope, defined in $\Omega^*$, and
by $u_{\Omega^*}$ the stress function of $\Omega^*$.
Let
$$
t=\sup\{s\in[0,1]\,:\, s\Omega^*\subset\Omega\}
$$
and let
$$
\xx\in\partial(t\Omega^*)\cap\partial\Omega\,.
$$
Notice that $t>0$ since $0\in\Omega$ and $\Omega^*$ is bounded (for $\Omega$ is bounded) while $t<1$ for $\Omega$ is not convex.
It is easily seen that
$$
u_t(x)=t^2 u_{\Omega^*}\left(\frac{x}{t}\right),
$$
solves \eqref{i1} with $t\Omega^*$ in place of $\Omega$, (that is $u_t$ is the stress function of $t\Om^*$) whence
$$
\begin{cases}
-\Delta w=0 & {\rm in }\;t\Om^*\\
w\geq0 & {\rm on }\;\p(t\Om^*)\\
w(\xx)=0,
\end{cases}
$$
where $w=u_{\Omega}-u_t$.
\newcommand{\disegnocvx}{(-10,0) to[out=90,in=180] (-7,5) to[out=0,in=180] (0,2) to[out=0,in=180] (4,5) to[out=0,in=90] (8,0) to[out=270,in=0] (2,-5) to[out=180,in=270] (-10,0)}
\begin{figure}[h]
\centering
\begin{tikzpicture}[x=3.5mm,y=3.5mm]
\draw[very thick, gray] \disegnocvx;
\draw[gray] (-10,0) node[above right]{$\Omega$};
\draw[very thick, dashed, gray] (-7,5)-- (4,5);
\begin{scope}[scale=0.4]
\draw[very thick] (-10,0) to[out=90,in=180] (-7,5) to[out=0,in=180] (4,5) to[out=0,in=90] (8,0) to[out=270,in=0] (2,-5) 
to[out=180,in=270] (-10,0);
\draw[thick, dashed] \disegnocvx;
\draw (2,-5)node[above]{$t\Omega^*$};
\end{scope}
\draw[->] (0,2) node[below]{$\xx$}--node[right]{$\nu$}(0,2.8);
\draw[->] (-2.8,2)node[below]{$x_0$}-- (-2.8,2.8)node[left]{$\nu$};
\draw[->] (1.6,2)node[below]{$x_1$}--(1.6,2.8)node[right]{$\nu$};
\end{tikzpicture}
\caption{$t=\sup\{ s\in[0,1]\ :\ s\Omega^*\subseteq\Omega\}$.}\label{Fig-proofconvexity}
\end{figure}
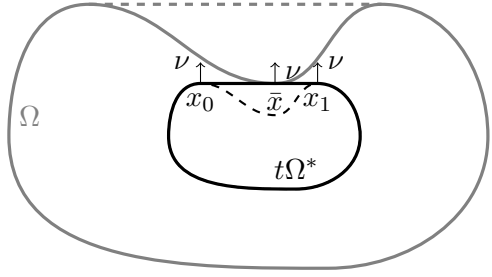

Then
\begin{equation}\label{hopf}
\frac{\partial w}{\partial\nu}(\xx)\leq0\,,
\end{equation}
where $\nu$ is the common outer unit normal at $\xx$ of $\Omega$ and $t\Omega^*$.
On the other hand
$$
\frac{\partial w}{\partial\nu}(\xx)=|\nabla u_t(\xx)|-|\nabla u_{\Omega}(\xx)|\,,
$$
then \eqref{hopf} reads
\begin{equation}\label{2becontradicted}
|\nabla u_t(\xx)|\leq |\nabla_{\Omega}u(\xx)|\,.
\end{equation}

We will contradict the latter by giving estimates of the gradient of $u_t$ which show that the reverse must be true.
To do this we need the following property, whose proof is given at the end of the section.

\begin{claim}\label{claim1}
If $x\in\partial\Omega^*\setminus\partial\Omega$, then
\begin{equation}\label{eqCLAIM1}
|\nabla u_{\Omega^*}(x)|\geq\left(\sum_{i=1}^N \la_i \sqrt{\nabla u(x_i)}\right)^2,
\end{equation}
where $x_1, ... x_N\in\partial\Omega$ are such that $x=\sum_{i=1}^N \la_ix_i$, for some $\la_1,...,\la_N\in[0,1]$ with $\sum_{i=1}^N\la_i=1$.
\end{claim}

With Lemma \ref{claim1} at hands, it is easy to prove that \eqref{2becontradicted}
can not hold true.
Indeed a straightforward calculation gives
\begin{equation}\label{nablaut}
\nabla u_t(x)=t\nabla u_{\Omega^*}(\frac{x}{t})\quad\text{for }x\in t\Omega^*.
\end{equation}

Notice that, as $\xx\in\bd t\Om^*$, there exist $x_1,...,x_N\in \bd t\Om$ and $\la_1,...,\la_N\in[0,1]$ such that $\sum_{i=1}^N\la_i=1$ and $\xx=\sum_{i=1}^N \la_ix_i$ and $\nu(\xx)||\nu(x_0)||\nu(x_1)||\nu$, where $\nu(x)$ indicates the outer unit normal vector to $t\Om^*$ at $x\in\bd t\Om$.
Moreover, since $\xx\in\partial t\Omega^*\setminus\partial\Omega$, we have
$$
\begin{array}{rl}
|\nabla u_t(\xx)|= t|\nabla u_{\Omega^*}(\xx/t)|
\geq& t\left(\sum_{i=1}^N \la_i\sqrt{|\nabla u(x_i/t)|}\right)^2\\
=&t\left(\la_i \sqrt{g(\frac{x_i}t)} \right)^2\\
=&t^{1-\alpha} \left( \sum_{i=1}^N \la_i \sqrt{g(x_i)} \right)^2\,.
\end{array}
$$
On the other hand, Remark \ref{rmkalphaconvexity} yields 
$$
g(\xx)\leq\left(\sum_{i=1}^N \la_i g(x_i)^{\frac 1{\alpha}}\right)^\alpha\,,
$$
and then
$$
g(\xx)\leq\left(\sum_{i=1}^N \la_i g(x_i)^{\frac 12}\right)^2,
$$
since $\alpha\geq 2$.
We finally get
$$
|\nabla u_t(\xx)|\geq t^{1-\alpha}g(\xx)>g(\xx)=|\nabla u(\xx)|\,,
$$
where the strict inequality holds because $t<1$, $\alpha>1$ and $g(\xx)>0$, since $\xx\neq 0$ for $0\in\Omega$.
The latter contradicts \eqref{2becontradicted} and the proof is concluded.
\end{proof}

We give the proof of Lemma \ref{claim1}.
\begin{proof}[{Proof of Lemma \ref{claim1}}]
We use the same notations and a similar construction as in the proof of Theorem \ref{propconvexity}.
So, let $\Om^*$ be the convex hull of $\Om$ and let $u_{\Om^*}, u_\Om$ be their stress functions, respectively.
By Comparison Principle we have
\begin{equation}\label{u*compar}
u_\Omega\leq u_{\Omega^*}\quad\text{in }\Omega\,.
\end{equation}
Since
$$
u_\Omega=u_{\Omega^*}\quad\text{on }\partial\Omega^*\cap\partial\Omega\,,
$$
equation \eqref{u*compar} entails
$$
|\nabla u_{\Omega^*}(x)|\geq|\nabla u_\Omega(x)|\quad\text{for every } x\in\partial\Omega^*\cap\partial\Omega\,.
$$
Fix $\xx\in\bd\Om^*$; there exist $x_1,...,x_N\in\bd\Om$ such that $\xx=\sum_{i=1}^N \la_ix_i$, with $\la_i\in[0,1], \sum \la_i=1$, and $Du_\Om(x_i)$ are parallel to $|Du_{\Om^*}(\xx)|$ for $i=1,...,N$ since the outer unit normal vector to $\Om^*$ at $\xx$ and to $\Om$ at $x_i$ must coincide. Let us set $\nu=\nu(\xx)=\nu(x_0)=\nu(x_1)$.
In particular we have
\begin{equation}\label{nablau*compar}
|\nabla u_{\Omega^*}(x_i)|\geq|\nabla u_\Omega(x_i)|\quad\text{for }i=1,...,N\,.
\end{equation}
Next we prove that
\begin{equation}\label{STEP2CLAIM1}
|\nabla u_{\Omega^*}(\xx)|\geq\left( \sum_{i=1}^N  \sqrt{|\nabla u_{\Omega^*}(x_i)|} \right)^2\,.
\end{equation}
By the convexity of $\Omega^*$ it is well known that $u_{\Omega^*}$ is $1/2$-concave (see \cite{CaffSp, Kenn, Kore}), that is
\begin{equation}\label{vconcave}
v=\sqrt{u_{\Omega^*}}\quad\text{is concave in }\overline\Omega^*\,.
\end{equation}
Moreover, since
$$
\sum_{i=1}^N\la_i(x_i-t\nu)=\xx-t\nu\,,
$$
for every small enough $t>0$ we have
$$
u_{\Omega^*}(\xx-t\nu)=v(\xx-t\nu)^2\geq \left(\sum_{i=1}^N\la_i v(x_i-t\nu)\right)^2\,.
$$
By the definition of $v$  the latter reads as
$$
\frac{u_{\Omega^*}(\xx-t\nu)}{t}\geq \left(\sum_{i=1}^N\la_i\sqrt{\frac{u_{\Omega^*}(x_i-t\nu)}{t}}\right)^2\,.
$$
Passing to the limit as $t\to 0^+$ and taking into account that
$$u_{\Omega^*}(\xx)=u_{\Omega^*}(x_i)=0\,, \quad i=1,...,N,
$$
we get \eqref{STEP2CLAIM1}.

Coupling \eqref{STEP2CLAIM1} and \eqref{nablau*compar} gives \eqref{eqCLAIM1} and the proof is concluded.
\end{proof}

\section{Symmetries}\label{symmetrysec}

In this section we prove that some symmetry properties of $g$ are inherited by the solution $\Om$.
In particular in Theorem \ref{theoradial} we consider the radial case, while Theorem \ref{theosteiner} treats the Steiner symmetric case.
In both cases the presented technique is based on rearrangements of sets and functions.

\begin{theorem}\label{theoradial}
If $g$ is radial and satisfies assumptions (\ref{assum_g}), then the solution $\Omega$ to the minimization problem (\ref{min1}) is a ball.
\end{theorem}
Notice that a radial solution can always be explicitly computed.
The above result states that in fact not radial solutions cannot exist.
Let us remark that for $\alpha>1$ this can also be seen as a straightforward corollary of Theorem \ref{propunique}.
\begin{proof}
Let us denote by $\Om^\#$ the ball of the same volume than $\Om$ and centre at $O$. 
Assume $\Omega\setminus\Omega^\#$ to be a set of positive measure.
We are going to reach a contradiction proving that the value $J(\Om^\#)$ is strictly better than the minimum value $J(\Omega)$.

Notice that, since $g$ is radial and increasing in each direction, we have 
$$
\Om^\#=\{x\in\rn \ :\ g(x)\le \bar t\},$$ 
for some $\bar t>0.$

\newcommand{\disegnoomega}{
(-3.5,0) to[out=90,in=200] (0,5) to[out=10,in=180] (2,5.2) to[out=0,in=90] (4,0) to[out=270,in=0] (0,-3) to[out=180,in=270] (-3.5,0)
}
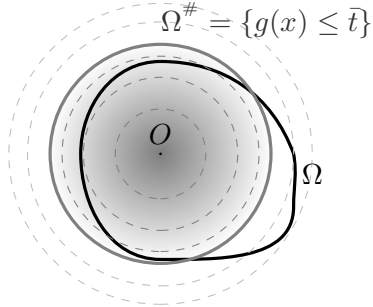
\begin{figure}[h]
\centering
\begin{tikzpicture}[x=3.5mm,y=3.5mm]
\begin{scope}[rotate=270]
\shade [shading=radial] (0,0) circle (4.15);
\fill[black] (0,0) circle (0.5pt);
\draw (0,0)node[above]{$O$};
\draw[very thick] (-3.5,0) to[out=90,in=200] (0,5) to[out=10,in=180] (2,5) to[out=0,in=90] (4,0) to[out=270,in=0] (0,-3) to[out=180,in=270] (-3.5,0);
\draw[very thick, gray] (0,0) circle (4.15);
\draw[dashed, gray] (0,0) circle (1.7); \draw[dashed, gray] (0,0) circle (2.9); 
\draw[dashed, gray] (0,0) circle (3.7); \draw[dashed, gray!50] (0,0) circle (5); \draw[dashed, gray!50] (0,0) circle (5.7); 
\end{scope}
\draw (5.7,0) node[below]{$\Omega$};
\draw[black!80] (4,4) node[above]{$\Omega^{^\#}=\{ g(x)\le \bar t \}$};
\end{tikzpicture}
\caption{The level lines of $g$ are monotone increasing concentric balls.
}
\end{figure}

Hence
$$
\inf_{x\in\Om\setminus\Om^\#} g^2(x) > \bar t=\sup_{x\in\Om^\#\setminus\Om} g^2(x),
$$
which implies $\phi(\Om^\#)< \phi(\Om)$, since
\begin{eqnarray*}
\phi(\Om^\#)&\le& \int_{\Om^\#\cap\Om}g^2(x)\,dx + |\Om^\#\setminus\Om|\sup_{x\in\Om^\#\setminus\Om} g^2(x)\\ 
&<& \int_{\Om^\#\cap\Om}g^2(x)\,dx + |\Om^\#\setminus\Om|\inf_{x\in\Om\setminus\Om^\#} g^2(x) \\
&\le& \int_{\Om^\#\cap\Om}g^2(x)\,dx + \int_{\Om\setminus\Om^\#}g^2(x)\,dx\\ 
&=& \phi(\Om).
\end{eqnarray*} 

In order to compare the stress function of $\Om$ with that of $\Om^\#$, we compare both of them with a rearrangement of $u_\Om$.
More precisely, let $u^\#$ be the Schwarz symmetric of the stress function $u_\Om$, that is the radial function whose sublevel sets are concentric balls of the same measure than the corresponding sublevel sets of $u$ (that is $|\{u^\#<t\}|=|\{u_\Om<t\}|$). 
We compare $u^\#$ with the stress functions of $\Om^\#$, $u_{\Om^\#}$ (notice that it is a radial function too) and that of $\Om$.
Recalling the characterization of stress functions in (\ref{i3}), we have
\begin{eqnarray*}
J(\Om^\#) 
 \le  \frac 12\int_{\Om^\#}\|Du^\# \|^2 \, dx - \int_{\Om^\#} {u^\#}^2 \, dx;
\end{eqnarray*}
moreover, by classical results (see for instance \cite{Kaw}), it holds
\begin{eqnarray*}
\int_{\Om^\#} {u^\#}^2 \, dx & = &  \int_{\Om} {u_\Om}^2 dx,\\
\int_{\Om^\#} \| Du^\# \|^2 \, dx & < & \int_{\Om} \| Du_{\Om}\|^2 dx,
\end{eqnarray*}
where the strict inequality holds since $u_{\Omega}$ is not radial (otherwise $\Omega$ would be a ball) and hence it does not coincides with $u^\#$ (see \cite{Kaw} Corollary 2.33).
These latters entail
$J(\Om^\#)< J(\Om)$, which is a contradiction. 
This shows that $\Omega$ is a ball, up to a zero measure set.
\end{proof}

Notice that, under stronger assumptions a similar result on the symmetry of the solution to the torsion problem has been proved by  A. Greco. in \cite{Greco1}.
More precisely he considered Problem (\ref{ovfbp1}) with $g(x)=c |x|$, and he proved that if a solution exists and the set $\Om$ contains the origin, then the set must be a ball. 

We want to use Theorem \ref{thmboundlevelsets} to estimate the stability of the radial setting. 
Roughly speaking we deal with the following questions: 
\emph{if, in some sense, $g$ is close to be a radial function, is $\Omega$ close (in a suitable sense) to be a ball? And how does the distance of $\Omega$ from the ball shape depend on the distance of $g$ from the radial shape?}

It is then necessary to specify which kind of {distance} is convenient to use to measure the closeness of $g$ to be radial.

We present two stability results which are in fact corollaries of Theorem \ref{teomonot}.
More precisely, we consider two different kinds of distances of functions: in Proposition \ref{corollstab} we ask $g$ to be \emph{quasi radial} in the $L^{\infty}$ norm, while in Proposition \ref{propstab} the distance between the function $g$ and a radial function is controlled in terms of the $\alpha$-homogeneous sublinear function $|x|^\alpha$. 
Notice that this latter is in fact quite natural since the space of $\alpha$-homogeneous functions is considered.

\begin{proposition}\label{corollstab}
Let $g, h$ satisfy (\ref{assum_g}) with $\alpha>1$ and assume $h$ to be radial with 
$$
| g(x)-h(x) |\le \ep,
$$ 
for every $x\in\rn$ and some $\ep>0$.
Let $\Omega$ be the solution of (\ref{ovfbp1}) related to $g$, then there exist $R_{\ep}>r_{\ep}>0$ such that  $B(O, r_{\ep})\subseteq\Omega\subseteq B(O, R_{\ep})$ with $|R_\ep-r_\ep|=\mathscr{O}(\ep)$.
\end{proposition}
\begin{proof}
Since the function $h$ is $\alpha$-homogeneous and radial, there exists a positive constant $k$ such that $h(x)=k|x|^{\alpha}$, where in fact $k=h|_{S^{N-1}}$. 
Define 
$$
h_-(x)=\frac{h(x)}{1-\ep},\qquad h_+(x)=\frac{h(x)}{1+\ep};
$$
it holds $h_+\le g\le h_-$ since
$$
\{ x\in\rn\ :\ h_+(x) \le 1 \}\subseteq \{x\in\rn\ :\ g(x)\le 1 \}\subseteq\{ x\in\rn\ :\ h_-(x)\le 1 \}.
$$
Moreover $h_-,h_+$ satisfy the hypothesis of Theorem \ref{teomonot} and hence 
$$
\Omega_{-}\subseteq \Omega \subseteq \Omega_{+},
$$
where $\Omega_{-},\Omega_{+}$ are the solutions to (\ref{ovfbp1}) related to $h_-,h_+$, respectively.

By Theorem \ref{theoradial} there exist $R_{\ep}>r_{\ep}>0$ such that $\Omega_{-}= B(O, r_{\ep})$ and $\Omega_+= B(O, R_{\ep})$.
In particular, solving explicitly Problem \ref{ovfbp1} in the radial homogeneous case, we get 
$$
R_{\ep}=\big(\frac{1+\ep}{kN}\big)^{\frac 1{\alpha-1}},\qquad r_{\ep}=\big(\frac{1-\ep}{kN}\big)^{\frac 1{\alpha-1}},
$$
since the stress function of a ball $B(O,\rho)$ is $u(x)=\frac{\rho^2-|x|^2}{2N}$, with $|Du||_{\partial B(O,\rho)}=\frac{\rho}N$.
Comparing $R_\ep$ and $r_\ep$ we have 
$$
|R_\ep-r_\ep|=\frac 2{(\alpha-1)(Nk)^{\frac 1{\alpha-1}}} (\ep+o(\ep)),
$$
which entails the thesis.
\end{proof}

\begin{proposition}\label{propstab}
Let $g, h$ satisfy (\ref{assum_g}) with $\alpha>1$ and assume $h$ to be radial with 
$$
| g(x)-h(x)| \le \ep|x|^{\alpha},
$$ 
for every $x\in\rn$ and some $\ep>0$.
Let $\Omega$ be the solution of (\ref{ovfbp1}) related to $g$. 
There exist $R'_{\ep}>r'_{\ep}>0$ such that  $B(O, r'_{\ep})\subseteq\Omega\subseteq B(O, R'_{\ep})$ with $|R'_{\ep}-r'_{\ep}|=\mathscr{O}(\ep)$.
\end{proposition}
\begin{proof}
As before we notice that the assumptions on $h$ implies that there exists a positive constant $k$ such that $h(x)=k|x|^{\alpha}$.

Since the stress function of a ball $B(O,\rho)$ is of the form  $u(x)=\frac 1{2N}(\rho-|x|^2)$, the solutions to (\ref{ovfbp1}) related to $h+\ep|x|^{\alpha},h-\ep|x|^{\alpha}$ are
$$
u_+(x)=\frac{{r'_{\ep}}^2-|x|^2}{2N}, \qquad u_-(x)=\frac{{R'_{\ep}}^2-|x|^2}{2N},
$$
with $r'_{\ep}={\big(N(k+\ep)\big)^{-\frac 1{\alpha-1}}}$, $R'_{\ep}={\big(N(k-\ep)\big)^{-\frac 1{\alpha-1}}}$, respectively.

Then, for every $\alpha>1$, it holds $\lim_{\ep\to 0^+} r'_{\ep}=\lim_{\ep\to 0^+} R'_{\ep}=(kN)^{-\frac 1{\alpha-1}}$, which entails the thesis.
Notice that the value $r=(kN)^{-\frac 1{\alpha-1}}$, corresponds to the value of the radius of the ball solution to Problem (\ref{ovfbp1}) related to the function $h(x)$.
\end{proof}

\begin{remark}
Notice that, by the homogeneity of $g$ and $h$, the condition $|g(x)-h(x)| \le \ep$ in Proposition \ref{corollstab} guarantees the validity of the condition $|g(x)-h(x)| \le\ep|x|^{\alpha}$ in Proposition \ref{propstab}.
Indeed 
$$
h(x)+\ep|x|^{\alpha}=|x|^{\alpha}(h(\frac x{|x|}) +\ep)\ge |x|^{\alpha}(g(\frac x{|x|}))=g(x),
$$
and, analogously, $h(x)-\ep\le g(x)$.

However Proposition \ref{corollstab} is stronger than Proposition \ref{propstab}, depending on the value of $k>0$.
In particular if $k< 1-\ep$, then the balls $B(O,r_\ep), B(O,R_\ep)$ give a better approximation to the set $\Omega$, while for $k>1+\ep$, it is more convenient to compare $\Omega$ with $B(O,r'_\ep), B(O,R'_\ep)$.
For $1-\ep<k<1+\ep$ we have $R_\ep<R'_\ep$ and $r_\ep>r'_\ep$.
\end{remark}

\begin{remark}
Since $g$ is homogeneous, it is completely described by its degree of homogeneity $\alpha$ and one of its level set, say
$$
G_1=\{x\in\rn\,:\,g(x)< 1\}\,.
$$
Indeed the sublevel set $G_t$ can be written in terms of $G_1$, see (\ref{GtG1}).

Hence the distance of $g$ to be radial can be conveniently expressed in terms of the distance of $G_1$ from a ball.

Propositions \ref{corollstab} and \ref{propstab} can in fact be rewritten in terms of sublevel sets.
More precisely the condition $|g-h|\le \ep$ becomes
$$
B\Big(O,\big(\frac{1-\ep}k\big)^{\frac 1{\alpha}}\Big)\subset G_1\subset B\Big(O,\big(\frac{1+\ep}k\big)^{\frac 1{\alpha}}\Big),
$$
where $k=h|_{S^{N-1}}$, while $|g-h|\le \ep|x|^{\alpha}$ entails
$$
B\Big(O,\frac 1{(k+\ep)^{\frac 1{\alpha}}}\Big)\subset G_1\subset B\Big(O,\frac 1{({k-\ep})^{\frac 1{\alpha}}}\Big).
$$
\end{remark}

We now present another symmetry result concerning the Steiner symmetric case.
More precisely, let us assume $g$ to be \emph{Steiner symmetric} in the following sense:
\begin{equation}\label{steiner1}
\begin{cases}
\text{$g$ is symmetric with respect to $\{x_N=0\}$ and} \\
\text{the sublevel sets of $g$ are convex in the $x_N$ direction.}   
\end{cases}
\end{equation}
Notice that this can be rephrased as 
$$
g(x',x_N)\ge g(x',y_N), \qquad\text{whenever }|x_N|\ge|y_N|.
$$

\begin{theorem}\label{theosteiner}
Consider a function $g$ satisfying assumptions (\ref{assum_g}) and assume $g$ to be Steiner symmetric in the sense of (\ref{steiner1}).
Then the solution $\Om$ to Problem (\ref{min1}) is symmetric with respect to the hyperplane $\{x_N=0\}$.
\end{theorem}
\begin{proof}
The proof makes use of Steiner symmetrization. 
Let us recall some notations; for more
details, we refer to \cite{PoSz}.

We denote by $\Om'$ the projection of $\Om$ onto $\R^{N-1}$:
$$
\Omega'=\{x'\in\R^{N-1}\ \text{such that there exists}\ x_N\ \text{with }(x',x_N)\in\Omega\}.
$$ 
For $x'\in \R^{N-1}$, we denote by $\Om(x')$ the intersection of $\Om$ with the line $\{x'\}\times\R$; that is
$$
\Omega(x'):=\{x_N\in\R\ \text{such that }(x',x_N)\in\Omega\}.
$$ 
Obviously $\Om(x')$ is the empty set for every $x'\in\R^{N-1}\setminus \Om'$ while 
$$
\bigcup_{x'\in\Om'} \left( x'\times\Om(x')\right) =\Om.
$$
We introduce the one-dimensional set 
$$
\Omega^\star(x'):=\left(-\frac{1}{2}|\Omega(x')|,\frac{1}{2}|\Omega(x')|
\right);
$$ 
which is a symmetric interval with the same measure as $\Om(x')$.

The \emph{Steiner symmetrized} of $\Om$ with respect to the hyperplane
$\{x_N=0\}$ is the set $\Om^\star$ defined by
$$
\Omega^{\star}:=\left\{ x=(x',x_N)\ \text{such that } -\frac{1}{2}|\Omega(x')| < x_N < \frac{1}{2}|\Omega(x')|,\;x'\in\Omega'\right\}.
$$ 
Now, since $g$ is increasing in the $x_N$ direction from $x_N=0$ and symmetric, we have
for any $x'\in\Om'$
\begin{eqnarray*}
    \int_{\Om(x')} g^2(x',x_N)dx_N & = & \int_{\Om(x')\cap\Omega^\star(x')} g^2(x',x_N) dx_N + \int_{\Om(x')\setminus\Omega^\star(x')} g^2(x',x_N) dx_N \\
    &\ge& \int_{\Om^\star(x')\cap\Omega(x')} g^2(x',x_N) dx_N + \int_{\Om^\star(x')\setminus\Omega(x')} g^2(x',x_N) dx_N \\
&=& \int_{\Om^\star(x')} g^2(x',x_N)dx_N.
  \end{eqnarray*}
Therefore, integrating over $\Om'$, we get
$$
\int_\Om g^2 dx=\int_{\Om'} dx'\int_{\Om(x')}g^2(x',x_N)dx_N \geq \int_{\Om^\star} g^2 dx;
$$
this shows that $\Om^\star$ is also admissible for the minimization problem (\ref{min1}).
Denote by $u^\star$ the Steiner symmetrization of $u_\Om$, that is the function whose level sets are
the Steiner symmetrization of level sets of $u_\Om$.
By Fubini Theorem and classical results on rearrangement, it is well known that
\begin{eqnarray*}
\int_{\Omega} u_\Om \,dx &=& \int_{\Omega^{\star}} u^\star \,dx ,\\ 
\int_{\Omega} |\nabla u_\Om|^2 \,dx &\geq& \int_{\Omega^{\star}} |\nabla u^\star|^2\,dx.
\end{eqnarray*}
Thus, since $u^\star$ belongs to the Sobolev space $H^1_0(\Om^\star)$, using the
variational characterization (\ref{i3}) we get
$$
J(\Om^\star)=G_{\Om^\star}(u_{\Om^\star})\leq G_{\Om^\star}(u^\star) \leq G_\Om(u_\Om)=J(\Om),
$$
and the proof is complete.
\end{proof}

\section{A relation with the level sets of $g$}\label{finalSec}
\medskip

As largely proved in the previous sections, clearly the geometry of $g$ influences the geometry of the solution $\Omega$. Following the radial case, one could expect the shape of $\Omega$ to be strongly related 
to the shape of the level sets of $g$. Indeed in the radial case,  $\Omega$ is an homothetic copy of the level sets of $g$; unfortunately this happens only in this particular case, as can be easily inferred from the Serrin result \cite{Serrin}. On the other hand, some estimate of the solution $\Omega$ in terms of $g$ is always possible and this is precisely the aim of this final section. Roughly speaking, we will show that
$\Omega$ must be incapsulated between two a priori known level sets of $g$.
To give a precise statement, it is convenient to introduce first some notations.

As before, assume $g$ to be homogeneous of degree $\alpha>0$ and $g(x)>0$ if $x\neq 0$.
Let $t\in(0,\infty)$ and denote by $G_t$ the (open) $t$-sublevel set of $g$, that is
$$
G_t=\{x\in\R^n\,:\,g(x)< t\}\,.
$$
By homogeneity, it is easily seen that all the level sets are dilatation of $G_1$, precisely
\begin{equation}\label{GtG1}
G_t=t^{\frac1{\alpha}}\,G_1\,.
\end{equation}
Now let $u_t$ be the stress function of $G_t$ and assume that $g$ is regular enough to get
$u_t\in C^2(G_t)\cap C^1(\overline G_t)$ (for instance $g\in C^{1,\beta}(\rn)$ for some $\beta>0$). 

In particular $u_1$ is the solution of
\begin{equation}\label{u1}
\begin{cases}
-\Delta u_1=1 & \qquad\text{in }\;G_1\\
u_1=0 & \qquad\text{on }\;\p G_1,
\end{cases}
\end{equation}
and it holds
\begin{equation}\label{utu1}
u_t(x)=t^{2/\alpha}\,u_1\left(\frac{x}{t^{1/\alpha}}\right)\,.
\end{equation}
Set
\begin{equation}\label{AB}
A=\min_{\partial G_1}|\nabla u_1|\,,\quad B=\max_{\partial G_1}|\nabla u_1|\,.
\end{equation}
Notice that $A$ and $B$ depends only on $g$ and they are, in principle, a priori known.
Moreover, $A\leq B$ and $A<B$ unless $G_1$ is a ball (see \cite{Serrin}), that is $g$ is radial; the radial case is 
discussed in detail in the previous section.

\begin{theorem}\label{thmboundlevelsets}
If $\alpha >1$ then
$$
A^{1/(\alpha-1)}G_1\subseteq\Omega\subseteq B^{1/(\alpha-1)}G_1\,.
$$
\end{theorem}
\begin{proof}
Since the origin $O$ is in the interior of  both $\Omega$ and $G_1$, and they are both bounded, there exist $r$ and 
$s$ such that $0<r\leq s$ and
$$
r=\sup\{t\,:\, G_t\subseteq\Omega\}\quad\text{and}\quad
s=\inf\{t\,:\, \Omega\subseteq G_t\}\,.
$$
Then $G_r\subseteq\Omega\subseteq G_s$ and there exist
$$
x_r\in\partial G_r\cap\partial\Omega\quad\text{and}\quad x_s\in\partial\Omega\cap\partial G_s\,.
$$
\newcommand{\Gr}{
(-6,0) to[out=90,in=200] (0,5) to[out=20,in=180] (5,6) to[out=0,in=90] (8,0) to[out=270,in=0] (0,-5) to[out=180,in=315] (-4,-4) to[out=135,in=270] (-6,0)
}
\newcommand{\disegnoomega}{
(-3.5,0) to[out=90,in=200] (0,5) to[out=10,in=180] (2,5.2) to[out=0,in=90] (4,0) to[out=270,in=0] (0,-3) to[out=180,in=270] (-3.5,0)
}
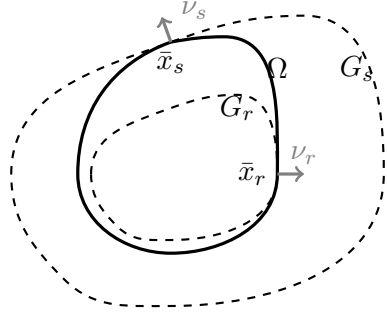
\begin{figure}[h]
\centering
\begin{tikzpicture}[x=3.5mm,y=3.5mm]
\draw[thick, dashed] \Gr;
\draw (7,4) node{$G_s$};
\begin{scope}[scale=0.5]
\draw[thick, dashed] \Gr;
\end{scope}
\draw (2.5,2.5) node{$G_r$};
\draw[very thick] \disegnoomega;
\draw (4,4) node{$\Omega$};
\draw[->, gray, very thick] (90:5)-- ++(110:1);
\draw[gray] (0,5.5) node[above right]{$\nu_s$};
\draw (0,5) node[below]{$\xx_s$};
\draw[->, gray, very thick] (0:4)-- ++(0:1);
\draw[gray] (5,0) node[above]{$\nu_r$};
\draw (4,0) node[left]{$\xx_r$};
\end{tikzpicture}
\caption{$G_r\subseteq\Omega\subseteq G_s$}\label{Fig-GrGs}
\end{figure}
We want to estimate $r$ and $s$ in terms of $g$.
Then let $w_r=u_\Omega-u_r$, where $u_\Omega$ is as usual the stress function of $\Om$, and notice that $w_r$ satisfies
$$
\begin{cases}
-\Delta w_r=0 &\qquad\text{in }\;G_r\\
w_r\geq0 & \qquad\text{on }\;\p G_r\\
w_r(x_r)=0. &
\end{cases}
$$
By maximum principle $w_r\geq 0$ in $G_r$ and then $x_r$ is a minimum point. 
Whence
$$
\frac{\partial w_r}{\partial\nu_r}(x_r)=\frac{\partial u_\Omega}{\partial\nu_r}(x_r)-\frac{\partial u_r}{\partial\nu_r}(x_r)=
|\nabla u_r(x_r)|-|\nabla u_\Omega(x_r)|\leq 0\,,
$$
where $\nu_r$ is the outer normal of $G_r$ (and $\Omega$) at $x_r$.
Since $\Omega$ solves Problem (\ref{i1}) and $x_r\in\partial G_r$, the latter reads
\begin{equation}\label{nablaur}
|\nabla u_r(x_r)|\leq g(x_r)=r\,.
\end{equation}
On the other hand, \eqref{utu1} gives
$$
|\nabla u_r(x_r)|=r^{1/\alpha}\left|\nabla u_1\left(\frac{x_r}{r^{1/\alpha}}\right)\right|
$$
and thanks to \eqref{GtG1} it holds
$$
\frac{x_r}{r^{1/\alpha}}\in\partial G_1\,.
$$
Then from \eqref{AB} we get
$$
r\geq r^{1/\alpha}A,
$$
which implies, for $\alpha>1$,
\begin{equation}\label{rA}
r\geq A^{\alpha/(\alpha-1)};
\end{equation}
and this proves the first inclusion of the statement, using \eqref{GtG1}.
\medskip

To obtain the second inclusion we argue in a similar way. Let $w_s=u_s-u_\Omega$ and notice that it solves
$$
\begin{cases}
-\Delta w_s=0 & \qquad\text{ in }\;\Omega\\
w_s\geq0 & \qquad\text{ on }\;\p \Omega\\
w_s(x_s)=0
\end{cases}
$$
Arguing as before we get
\begin{equation}\label{nablaus}
|\nabla u_s(x_s)|\geq g(x_s)=s\,.
\end{equation}
Coupling the latter with \eqref{utu1} and taking again into account \eqref{GtG1} and \eqref{AB}, we get
$$
s\leq s^{1/\alpha}B,
$$
whence, if $\alpha>1$, we obtain
\begin{equation}\label{rA2}
s\leq B^{\alpha/(\alpha-1)},
\end{equation}
which proves the second inclusion of the statement, using again \eqref{GtG1}.
\end{proof}
\section*{Acknowledgements}
The authors wish to warmly thank Dorin Bucur for several helpful discussions.

\smallskip

This paper started while Chiara Bianchini was supported by the project ANR-09-BLAN-0037 {\it Geometric analysis of optimal shapes (GAOS)} and by the research group INRIA {\it Contr\^ole robuste infini-dimensionnel et applications (CORIDA)}.

The work of Antoine Henrot is part of the projects ANR-09-BLAN-0037 {\it Geometric analysis of optimal shapes (GAOS)} and ANR-12-BS01-0007-01-OPTIFORM {\it Optimisation de formes} financed by the French Agence Nationale de la Recherche (ANR).

The work of Chiara Bianchini and Paolo Salani is part of the projects
GNAMPA 2012 {\it Problemi sovradeterminati e geometria delle soluzioni per
equazioni ellittiche e paraboliche} and PRIN 2008 {\it Aspetti geometrici
delle equazioni alle derivate parziali e questioni connesse}.

\end{document}